\documentclass{amsart}
\usepackage{amsmath}
\usepackage{amsthm,amscd,amssymb,amsfonts, amsbsy}
\usepackage{latexsym}
\usepackage{txfonts, mathrsfs}

\usepackage{color}
\usepackage{enumerate}

\numberwithin{equation}{section}

\theoremstyle{plain}
\newtheorem{theorem}[equation]{Theorem}
\newtheorem{lemma}[equation]{Lemma}
\newtheorem{corollary}[equation]{Corollary}

\theoremstyle{definition}
\newtheorem{definition}[equation]{Definition}

\newtheorem{example}[equation]{Example}
\newtheorem*{acknowledgment}{Acknowledgment}
\newtheorem*{CLB}{Condition (LB)}
\newtheorem*{CIH}{Condition (IH)}
\newtheorem*{CLH}{Condition (LH)}
\newtheorem*{CS}{Condition (S)}

\theoremstyle{remark}
\newtheorem{remark}[equation]{Remark}

\newtheorem{notation}[equation]{Notation}

\newcommand{\dv}{\operatorname{div}}

\newcommand{\supp}{\operatorname{supp}}
\newcommand{\dist}{\operatorname{dist}}
\newcommand{\diam}{\operatorname{diam}}

\newcommand{\mysection}[1]{\section{#1}
\setcounter{equation}{0}}

\renewcommand{\vec}[1]{\boldsymbol{#1}}

\newcommand{\bR}{\mathbb R}

\newcommand{\bH}{\mathbb H}

\newcommand\Lt{{}^t\!L}

\newcommand{\ip}[1]{\left\langle#1\right\rangle}

\providecommand{\bigset}[1]{\bigl\{#1\bigr\}}

\providecommand{\abs}[1]{\lvert#1\rvert}
\providecommand{\Abs}[1]{\left\lvert#1\right\rvert}
\providecommand{\bigabs}[1]{\bigl\lvert#1\bigr\rvert}
\providecommand{\Bigabs}[1]{\Bigl\lvert#1\Bigr\rvert}

\providecommand{\Biggabs}[1]{\Biggl\lvert#1\Biggr\rvert}

\providecommand{\norm}[1]{\lVert#1\rVert}

\providecommand{\bignorm}[1]{\bigl\lVert#1\bigr\rVert}

\renewcommand{\epsilon}{\varepsilon}
\renewcommand{\qedsymbol}{$\blacksquare$}

\begin{document}
\title[Green's matrix]
{Global pointwise estimates for Green's matrix of second order elliptic systems}

\author[K. Kang]{Kyungkeun Kang}
\address[K. Kang]{Department of Mathematics, Sungkyunkwan University, Suwon 440-746, Republic of Korea}
\email{kkang@skku.edu}

\author[S. Kim]{Seick Kim}
\address[S. Kim]{Department of Mathematics, Yonsei University, Seoul 120-749, Republic of Korea}
\curraddr{Department of Computational Science and Engineering, Yonsei University, Seoul 120-749, Republic of Korea}
\email{kimseick@yonsei.ac.kr}



\begin{abstract}
We establish global pointwise bounds for the Green's matrix for divergence form, second order elliptic systems in a domain under the assumption that weak solutions of the system vanishing on a portion of the boundary satisfy a certain local boundedness estimate.
Moreover, we prove that such a local boundedness estimate for weak solutions of the system is equivalent to the usual global pointwise bound for the Green's matrix.
In the scalar case, such an estimate is a consequence of De Giorgi-Moser-Nash theory and holds for equations with bounded measurable coefficients in arbitrary domains.
In the vectorial case, one need to impose certain assumptions on the coefficients of the system as well as on domains to obtain such an estimate.
We present a unified approach valid for both the scalar and vectorial cases and discuss several applications of our result.
\end{abstract}

\maketitle

\mysection{Introduction}        \label{intro}

In this article, we are concerned with the Green's matrix for elliptic systems
\begin{equation}    \label{eq0.0}
\sum_{j=1}^m L_{ij} u^j:= -\sum_{j=1}^m \sum_{\alpha,\beta=1}^ n D_{\alpha}\Bigl(A^{\alpha\beta}_{ij} D_\beta u^j\Bigr), \quad i=1,\ldots, m,
\end{equation}
in a (possibly unbounded) domain $\Omega\subset \bR^n$, $n\geq 3$. We assume that the coefficients are measurable functions defined in the whole space $\bR^n$ satisfying the strong ellipticity condition
\begin{equation}    \label{eqP-02}
\sum_{i,j=1}^m \sum_{\alpha,\beta=1}^n A^{\alpha\beta}_{ij}(x)\xi^j_\beta \xi^i_\alpha \geq \nu\sum_{i=1}^m\sum_{\alpha=1}^n \bigabs{\xi^i_\alpha}^2 =:\nu \bigabs{\vec \xi}^2, \quad\forall \vec\xi \in \bR^{mn},\quad\forall x\in\bR^n,
\end{equation}
and also the uniform boundedness condition
\begin{equation}    \label{eqP-03}
\sum_{i,j=1}^m \sum_{\alpha,\beta=1}^n \Bigabs{A^{\alpha\beta}_{ij}(x)}^2\le \nu^{-2},\quad\forall x\in\bR^n,
\end{equation}
for some constant $\nu\in (0,1]$.
We do not assume the the coefficients are symmetric.
We will later impose some further assumptions on the operator but not explicitly on its coefficients.

In the scalar case (i.e. $m=1$), the Green's matrix becomes a scalar function and is usually called the Green's function.
It is well known that the Green's function $G(x,y)$ is nonnegative in $\Omega$ and for all $x, y \in \Omega$ with $x\neq y$, we have
\begin{equation}        \label{eq0.1}
G(x,y)\leq K \abs{x-y}^{2-n},
\end{equation}
where $K$ is a constant depending on the dimension $n$ and the ellipticity constant  $\nu$ of the operator; see \cite{LSW}, \cite{GW}.
It is also known that if $\Omega$ is a bounded domain satisfying the uniform exterior cone condition, then for all $x, y\in\Omega$ with $x\neq y$, we have
\begin{equation}        \label{eq0.2}
G(x,y)\leq K d^\alpha_y \abs{x-y}^{2-n-\alpha};\quad d_y=\dist(y,\partial\Omega),
\end{equation}
where $K$ and $\alpha\in (0,1)$ are constants depending only on $n, \nu$, and $\Omega$; see \cite{GW}.
The methods used in \cite{LSW} and \cite{GW} rely heavily on the Harnack's inequality and the maximum principle and does not work for the general vectorial case.
By assuming that  $\Omega$ is a bounded $C^1$ domain and the coefficients of the operator are uniformly continuous in $\overline\Omega$ (or belong to VMO), Dolzmann and M\"uller \cite{DM} proved the global estimate \eqref{eq0.1} in the vectorial setting.
It should be noted that Fuchs \cite{Fuchs84, Fuchs86} obtained a similar result earlier, but under a stronger assumption that the coefficients are H\"older continuous.
Recently, Hofmann and Kim \cite{HK07} derived the existence of a Green's matrix in an arbitrary domain, under the assumption that weak solutions of the system satisfy interior H\"older continuity estimates.
They also derived various estimates for the Green's matrix of such a system, including an interior version of the estimate \eqref{eq0.1}, which was applied to the development of the layer potential method for equations with complex coefficients in \cite{AAAHK}.
Their method is interesting because it works for both scalar and vectorial cases, but however, they did not attempt to derive the global estimates \eqref{eq0.1} or \eqref{eq0.2} for the Green's matrix in the vectorial setting.

The goal of this article is to present how one can derive a global estimate corresponding to \eqref{eq0.1} for Green's matrix of the elliptic systems \eqref{eq0.0} in a domain $\Omega$ using a local boundedness estimate for the weak solutions of the system vanishing on a portion of the boundary; see Condition~\eqref{LB} below for the precise statement of the local boundedness estimate.
In fact, we show that such a local boundedness estimate is a necessary and sufficient condition for the Green's matrix of the system to have a global pointwise bound like \eqref{eq0.1}.
We will also show how to derive a global estimate like \eqref{eq0.2} for Green's matrix of the elliptic system \eqref{eq0.0} in a domain $\Omega$ by using a local H\"older continuity estimate for the solutions of the system vanishing on a portion of the boundary $\partial\Omega$; see Condition~ \eqref{LH} below for the statement of the local H\"older estimate, and also the condition ~\eqref{eqP-09} in Remark~\ref{rmk:P-03}, which is a little bit weaker.
The novelty of our work is in presenting a unifying method that re-proves the global estimates \eqref{eq0.1} and \eqref{eq0.2} for the Green's function for the uniformly elliptic operators with bounded measurable coefficients as well as the corresponding estimates for the Green's matrix of the elliptic systems \eqref{eq0.0}, for instance, in a bounded $C^1$ domain with uniformly continuous or VMO coefficients.
Moreover, it has other interesting applications to $L^\infty$-perturbation of diagonal systems in a domain satisfying the uniform exterior cone condition, elliptic systems satisfying Legendre-Hadamard condition in a bounded $C^1$ domain with principal coefficients in VMO and lower order terms in $L^\infty$, Stokes systems in a three-dimensional Lipschitz domain, etc.; see Section~\ref{sec:app} below.
As a matter of fact, application to $L^\infty$-perturbation of diagonal systems in a domain satisfying the uniform exterior cone condition shows the power and flexibility of our method since that result does not seem to follow from other known methods, such as that based on the $L^p$ theory by Dolzmann and M\"uller \cite{DM}.

The organization of the paper is as follows. In Section~\ref{pre}, we introduce some notations and definitions including our definition of the Green's matrix of the system \eqref{eq0.0} in $\Omega$. In Section~\ref{main}, we give precise statement of the conditions concerning the local estimates for weak solutions of the systems and state our main theorems.  In Section~\ref{sec:app}, we present applications of our main results. The proofs of our main results are given in Section~\ref{sec:p} and a technical lemma is proved in Appendix.

Finally, a few remarks are in order. We do not treat the case $n=2$ in our paper. 
In two dimension, the Green's matrix has logarithmic growth and requires some different methods; see \cite{DM} and also \cite{DK09}.
As a matter of fact, the method used in this paper breaks down and does not work in that case.
By this reason, the two dimensional case will be discussed in a separate paper \cite{CDK}, where we will also treat a parabolic extension of our results.
As alluded earlier, the main difference between our result and that of \cite{HK07} is that they were mostly concerned with the Green's matrix of a $L^\infty$-perturbed diagonal systems in the whole space $\bR^n$ and focused on interior estimates of the Green's matrix while our paper is mainly concerned with the global estimates like \eqref{eq0.1} and \eqref{eq0.2}, which we believe are quite more useful in practice, especially in the vectorial case.
In \cite{AT},  Auscher and Tchamitchian introduced the ``Dirichlet property (D)'' in connection with the Gaussian estimates for the heat kernel of the operator $L$, which is very similar to the condition \eqref{LH} of this article.
We would like to hereby thank Pascal Auscher for kindly informing us about the paper \cite{AT}.

\mysection{Notations and Definitions}       \label{pre}

Let $L$ be an elliptic operator acting on column vector valued functions $\vec u=(u^1,\ldots,u^m)^T$ defined on a domain $\Omega\subset \bR^n$, $n\ge 3$, in the following way:
\[
L\vec u = -D_\alpha \bigl(\vec A^{\alpha\beta}\, D_\beta \vec u\bigr),
\]
where we use the usual summation convention over repeated indices $\alpha,\beta=1,\ldots, n$, and $\vec A^{\alpha\beta}=\vec{A}^{\alpha\beta}(x)$ are $m\times m$ matrix valued functions on $\bR^n$ with entries $A^{\alpha\beta}_{ij}$ that satisfy the conditions \eqref{eqP-02} and \eqref{eqP-03}.
Notice that the $i$-th component of the column vector $L \vec u$ coincides with $\sum_{j} L_{ij} u^j$ in \eqref{eq0.0}.
The adjoint operator $\Lt$ of $L$ is defined by
\[\Lt \vec u = -D_\alpha \bigl({}^t\!\vec A^{\alpha\beta} D_\beta \vec u\bigr),\]
where ${}^t\!\vec A^{\alpha\beta}=(\vec A^{\beta\alpha})^T$; i.e., ${}^t\!A^{\alpha\beta}_{ij}=A^{\beta\alpha}_{ji}$.
We use the same function space $Y^{1,2}(\Omega)$ as in \cite{HK07}. For reader's convenience, we reproduce the definition below.

\begin{definition}
For an open set $\Omega\subset\bR^n$ ($n\ge 3$), the space $Y^{1,2}(\Omega)$ is defined as the family of all weakly differentiable functions $u\in L^{2n/(n-2)}(\Omega)$, whose weak derivatives are functions in $L^2(\Omega)$.
The space $Y^{1,2}(\Omega)$ is endowed with the norm
\[
\norm{u}_{Y^{1,2}(\Omega)}:=\norm{u}_{L^{2n/(n-2)}(\Omega)}+\norm{D u}_{L^2(\Omega)}.
\]
We define $Y^{1,2}_0(\Omega)$ as the closure of $C^\infty_c(\Omega)$ in $Y^{1,2}(\Omega)$,
where $C^\infty_c(\Omega)$ is the set of all infinitely differentiable  functions with compact supports in $\Omega$.
\end{definition}

\begin{remark}          \label{rmk2.4q}
If $\abs{\Omega}<\infty$, then H\"older's inequality implies $Y^{1,2}(\Omega)\subset W^{1,2}(\Omega)$.
In the case $\Omega = \mathbb{R}^n$, we have  $Y^{1,2}(\bR^n)=Y^{1,2}_0(\bR^n)$.
Notice that by the Sobolev inequality, it follows that
\begin{equation}
\label{eqP-14}
\norm{u}_{L^{2n/(n-2)}(\Omega)} \le C(n) \norm{D u}_{L^2(\Omega)},\quad
\forall u\in Y^{1,2}_0(\Omega).
\end{equation}
Therefore, we have $W^{1,2}_0(\Omega)\subset Y^{1,2}_0(\Omega)$ and $W^{1,2}_0(\Omega)=Y^{1,2}_0(\Omega)$ if $\abs{\Omega}<\infty$; see \cite{MZ}.
\end{remark}

\begin{definition}
Let $\varSigma \subset \overline \Omega$ and $u$ be a $Y^{1,2}(\Omega)$ function.
We say that $u$ vanishes (or write $u=0$) on $\varSigma$ if $u$ is a limit in $Y^{1,2}(\Omega)$ of a sequence of functions in $C^\infty_c(\overline\Omega\setminus \varSigma)$.
\end{definition}

\begin{notation}
We denote by $L^\infty_c(\Omega)$  the family of all $L^\infty(\Omega)$ functions with compact supports in $\overline\Omega$. Notice that $L^\infty_c(\Omega)=L^\infty(\Omega)$ if $\Omega$ is bounded.
\end{notation}

\begin{notation}
We denote $\Omega_R(x)=\Omega\cap B_R(x)$ and $\varSigma_R(x)=\partial\Omega\cap B_R(x)$ for any $R>0$.
We abbreviate $\Omega_R=\Omega_R(x)$ and $\varSigma_R=\varSigma_R(x)$ if the point $x$ is well understood in the context.
\end{notation}

\begin{definition}      \label{def2}
We say that an $m\times m$ matrix valued function $\vec G(x,y)$, with entries $G_{ij}(x,y)$  defined on the set $\bigset{(x,y)\in\Omega\times\Omega: x\neq y}$, is a Green's matrix of $L$ in $\Omega$ if it satisfies the following properties:
\begin{enumerate}[i)]
\item
$\vec G(\cdot,y) \in W^{1,1}_{loc}(\Omega)$ and $L\vec G(\cdot,y)=\delta_y I$ for all $y\in\Omega$, in the sense that
\begin{equation}        \label{eq2.6r}
\int_{\Omega}A^{\alpha\beta}_{ij} D_\beta G_{jk}(\cdot,y) D_\alpha \phi^i = \phi^k(y),\quad
\forall \vec \phi=(\phi^1\ldots,\phi^m)^T\in C^\infty_c(\Omega).
\end{equation}
\item
$\vec G(\cdot,y) \in Y^{1,2}(\Omega\setminus B_r(y))$ for all $y\in\Omega$ and $r>0$ and $\vec G(\cdot,y)$ vanishes on $\partial\Omega$.
\item
For any $\vec f=(f^1,\ldots, f^m)^T \in L^\infty_c(\Omega)$, the function $\vec u$ given by
\begin{equation}        \label{eq2.9x}
\vec u(x):=\int_\Omega \vec G(y,x) \vec f(y)\,dy
\end{equation}
belongs to $Y^{1,2}_0(\Omega)$ and satisfies $\Lt \vec u=\vec f$ in the sense that
\begin{equation}        \label{eq2.7i}
\int_\Omega A^{\alpha\beta}_{ij} D_\alpha u^i D_\beta \phi^j = \int_\Omega f^j \phi^j, \quad\forall\vec \phi=(\phi^1,\ldots,\phi^m)^T\in C^\infty_c(\Omega).
\end{equation}
\end{enumerate}
\end{definition}

We note that part iii) of the above definition gives the uniqueness of a Green's matrix; see  \cite{HK07}.
We shall hereafter say that $\vec G(x,y)$ is ``the'' Green's matrix of $L$ in $\Omega$ if it satisfies all the above properties.

\mysection{Main results}						\label{main}
The following condition, which hereafter shall be referred to as \eqref{LB}, is used to obtain pointwise bounds for the Green's matrix $\vec G(x,y)$ of $L$ in $\Omega$.

\begin{CLB}
There exist $R_{max}\in (0,\infty]$ and $N_0>0$ such that for all $x\in\Omega$,  $R\in (0,R_{max})$, and $\vec f \in L^\infty(\Omega_R(x))$, the following holds:
If $\vec u \in W^{1,2}(\Omega_R(x))$ is a weak solution of either $L \vec u=\vec f$ or $\Lt \vec u=\vec f$ in $\Omega_R(x)$ and vanishes on $\varSigma_R(x)$, then we have
\begin{equation*}
\tag{LB}\label{LB}
\norm{\vec u}_{L^\infty(\Omega_{R/2})} \le N_0\left(R^{-n/2} \norm{\vec u}_{L^2(\Omega_R)}+ R^2 \norm{\vec f}_{L^\infty(\Omega_R)} \right);\quad \Omega_R:=\Omega_R(x).
\end{equation*}
\end{CLB}

\begin{notation}
We use the convention $c\cdot\infty=\infty$ for $c>0$ and $1/\infty=0$.
\end{notation}

\begin{remark}								\label{rmk3.3b}
By using a standard covering argument, it is easy to see that the constant $R_{max}$ in the condition \eqref{LB} is interchangeable with $c\cdot R_{max}$ for any fixed $c \in (0,\infty)$, possibly at the cost of changing the constant $N_0$ in the condition \eqref{LB} by $K\cdot N_0$, where $K=K(n,c)>0$.
\end{remark}

\begin{theorem}							\label{thm1}
Assume the condition \eqref{LB} and let $\vec G(x,y)$ be the Green's  matrix of $L$ in $\Omega$.
Then we have
\begin{equation}							\label{eq2.17d}
\abs{\vec G(x,y)} \leq C \abs{x-y}^{2-n}\quad \text{for all }\, x,y\in\Omega\; \text{ satisfying }\;0<\abs{x-y}<R_{max},
\end{equation}
where $C=C(n,m,\nu, N_0)$.
\end{theorem}

The following condition \eqref{IH} combines two conditions that appeared as the properties $\mathrm{(H)}$ and $\mathrm{(H)}_{loc}$ in \cite{HK07}. It means that weak solutions of $L\vec u=0$ and $\Lt \vec u=0$ in $B\subset \Omega$ are locally H\"older continuous  in $B$ with an exponent $\mu_0$.

\begin{notation}
We denote $a\wedge b=\min(a,b)$ and $d_x=\dist(x,\partial \Omega)$.
\end{notation}

\begin{CIH}
There exist $\mu_0\in (0,1]$, $R_c\in (0,\infty]$, and $N_1>0$ such that for all $x\in\Omega$ and $R \in(0, d_x\wedge R_c)$, the following holds:
If $\vec u\in W^{1,2}(B_R(x))$ is a weak solution of either $L\vec u=0$ or $\Lt \vec u=0$ in $B_R(x)$,  then we have
\begin{equation}
\tag{IH}\label{IH}
\int_{B_r(x)} \abs{D \vec u}^2 \le N_1\left(\frac{r}{s}\right)^{n-2+2\mu_0} \int_{B_s(x)} \abs{D \vec u}^2 \quad \text{for }\, 0<r<s\le R.
\end{equation}
\end{CIH}

\begin{theorem}     \label{thm1b}
Assume the conditions \eqref{IH} and \eqref{LB}. Then, the Green's matrix $\vec G(x,y)$ of $L$ in $\Omega$ exists and satisfies the estimate \eqref{eq2.17d} with $C=C(n,m,\nu, N_0)$.
Also, the Green's matrix ${}^t\vec G(x,y)$ of the adjoint operator $\Lt$ in $\Omega$ exists and we have
\begin{equation}        \label{eq2.20f}
\vec G(x,y)={}^t\vec G(y,x)^T,\quad \forall x,y\in\Omega,\quad x\neq y.
\end{equation}
Moreover, the Green's matrix $\vec G(x,y)$ satisfies the estimate
\begin{equation}        \label{eq3.9c}
\norm{\vec G(\cdot,y)}_{Y^{1,2}(\Omega\setminus B_r(y))}  \le C r^{(2-n)/2},\quad \forall r \in (0,R_{max}),\quad \forall y\in\Omega,
\end{equation}
where $C=C(n,m,\nu, N_0)$.
\end{theorem}

The following theorem says that the converse of Theorem~\ref{thm1} is also true if we assume the condition \eqref{IH}.
\begin{theorem}         \label{thm1c}
Assume the condition \eqref{IH} and let $\vec G(x,y)$ be the Green's matrix of $L$ in $\Omega$.
Suppose there exists $R_{max}\in (0,\infty]$ such that for all $x,y\in \Omega$ satisfying $0<\abs{x-y}<R_{max}$, we have
\begin{equation}							\label{eq2.17dd}
\abs{\vec G(x,y)} \leq C_0 \abs{x-y}^{2-n}.
\end{equation}
Then the condition \eqref{LB} is satisfied with the same $R_{max}$ and $N_0=N_0(n,m,\nu, C_0)$.
\end{theorem}

\begin{remark}							\label{rmk:P-03}
In fact, one can replace the condition \eqref{IH} in Theorem~\ref{thm1b} and Theorem~\ref{thm1c} by the following condition \eqref{eqP-09a}:
There exist $\mu_0\in (0,1]$, $R_c\in (0,\infty]$, and $C_0>0$ such that  if $\vec u\in W^{1,2}(B_R(x))$ is a weak solution of either $L\vec u=0$ or $\Lt \vec u=0$ in $B_R(x)$, where $x\in\Omega$ and $0<R<d_x\wedge R_c$, then we have
\begin{equation}						\label{eqP-09a}			\tag{IH${}^\prime$}
[\vec u]_{C^{\mu_0}(B_{R/2}(x))} \leq C_0 R^{-\mu_0}\left(\fint_{B_R(x)} \abs{\vec u}^2\right)^{1/2}.
\end{equation}
Here, $[\vec u]_{C^{\mu_0}(B_{R/2})}$ denotes the usual H\"older seminorm.
It is not hard to see that condition \eqref{IH} implies \eqref{eqP-09a} with $C_0=C_0(n,m,\nu,\mu_0,N_1)$ and  the same $\mu_0$ and $R_c$.
As a matter of fact, the conditions \eqref{IH} and \eqref{eqP-09a} are equivalent under our basic assumptions on the operator $L$; see \cite[Lemma~2.3]{HK07}.
However, the condition \eqref{eqP-09a} does not imply \eqref{IH}, for instance, in the presence of lower order terms in the operator $L$.
In this sense, \eqref{eqP-09a} is a weaker condition.
We point out that the properties (H) and (H)$_{loc}$ in \cite{HK07} can be replaced entirely by the condition \eqref{eqP-09a}, without affecting the conclusion of the main theorems in \cite{HK07}.
\end{remark}

The following condition, which hereafter shall be referred to as \eqref{LH}, is a combination of \eqref{IH} and another condition that appeared as the property $\mathrm{(BH)}$ in \cite{HK07}.

\begin{CLH}
There exist $\mu_0\in (0,1]$, $R_{max} \in (0,\infty]$, and  $N_1>0$ such that for all $x\in \Omega$ and $R \in (0,R_{max})$, the following holds:
If $\vec u\in W^{1,2}(\Omega_R(x))$ is a weak solution of either $L\vec u=0$ or $\Lt \vec u=0$ in $\Omega_R(x)$ and vanishes on $\varSigma_R(x)$, then we have
\begin{equation}
\tag{LH}\label{LH}
\int_{\Omega_r(x)} \abs{D \vec u}^2 \le N_1\left(\frac{r}{s}\right)^{n-2+2\mu_0} \int_{\Omega_s(x)} \abs{D \vec u}^2\quad \text{for }\, 0<r<s\le R.
\end{equation}
\end{CLH}

\begin{remark}								\label{rmk3.12}
In the condition \eqref{LH}, the constant $R_{max}$ is interchangeable with $c\cdot R_{max}$ for any fixed $c \in (0,\infty)$, possibly at the cost of changing the constant $N_1$ in the condition \eqref{LH} by $K\cdot N_1$, where $K=K(n,c)>0$.
\end{remark}

It will be shown in Appendix that \eqref{LH} implies \eqref{LB}.
Also, it is obvious that \eqref{LH} implies \eqref{IH}.
Therefore if \eqref{LH} is satisfied, then by Theorem~\ref{thm1b}, the Green's matrix $\vec G(x,y)$ of $L$ in $\Omega$ exists and satisfies the estimate \eqref{eq2.17d}. The following theorem asserts that in fact, in such a case, a better estimate for $\vec G(x,y)$ is available near the boundary $\partial\Omega$.

\begin{theorem}         \label{thm2}
Assume the condition \eqref{LH} and let $\vec G(x,y)$ be the Green's matrix of $L$ in $\Omega$.
Then for all $x,y\in \Omega$ satisfying $0<\abs{x-y}< R_{max}$, we have
\begin{equation}            \label{eq2.26s}
\abs{\vec G(x,y)}  \leq C \bigset{d_x\wedge \abs{x-y}}^{\mu_0} \bigset{d_y\wedge \abs{x-y}}^{\mu_0} \abs{x-y}^{2-n-2\mu_0},
\end{equation}
where $C=C(n,m,\nu, \mu_0,N_1)$.
If $R_{max}<\infty$ and $\Omega$ is bounded, then for all $x, y\in \Omega$ with $x\neq y$, we have the estimate \eqref{eq2.26s} with $C=C(n,m,\nu, \mu_0,N_1, R_{max}/\diam\Omega)$.
\end{theorem}

\begin{remark}							\label{rmk3.13a}
It will be clear from the proof of Theorem~\ref{thm2} (see \eqref{eq5.42a} in \S\ref{ss5.4}) that  we in fact have the following estimate, which is slightly stronger than \eqref{eq2.26s}: For all $x,y \in \Omega$ with $x\neq y$, we have
\[
\abs{\vec G(x,y)}  \leq C \bigset{d_x\wedge \abs{x-y}\wedge R_{max}}^{\mu_0} \bigset{d_y\wedge \abs{x-y}\wedge R_{max}}^{\mu_0} \bigset{\abs{x-y}\wedge R_{max}}^{2-n-2\mu_0},
\]
where $C=C(n,m,\nu, \mu_0,N_1)$.
\end{remark}

\begin{remark}							\label{rmk:P-04}
The following condition \eqref{eqP-09} can be used as a substitute for \eqref{LH} in Theorem~\ref{thm2}:
There exist $\mu_0\in (0,1]$, $R_{max} \in (0,\infty]$, and  $C_0>0$ such that for all $x\in \Omega$ and $R \in (0,R_{max})$, the following holds:
If $\vec u\in W^{1,2}(\Omega_R(x))$ is a weak solution of either $L\vec u=0$ or $\Lt \vec u=0$ in $\Omega_R(x)$ and vanishes on $\varSigma_R(x)$, then we have
\begin{equation}						\label{eqP-09}			\tag{LH${}^\prime$}
[\tilde{\vec u}]_{C^{\mu_0}(B_{R/2}(x))}  \leq C_0 R^{-\mu_0}\left(\fint_{B_R(x)} \abs{\tilde{\vec u}}^2\right)^{1/2},\quad \text{where}\;\: \tilde{\vec u}=\chi_{\Omega_R(x)} \vec u.
\end{equation}
It is not hard to verify that the condition~\eqref{LH} implies \eqref{eqP-09} with $C_0=C_0(n,m,\nu,\mu_0,N_1)$ and  the same $\mu_0$ and $R_{max}$.
Also, it can be easily seen that the condition \eqref{eqP-09} implies both the conditions \eqref{LB} and  \eqref{eqP-09a}.
We note that  the condition \eqref{eqP-09} is, however,  weaker than condition \eqref{LH} in general.
From the proof of Theorem~\ref{thm2}, it should be clear that the conclusion of Theorem~\ref{thm2} remains the same under a weaker condition \eqref{eqP-09}.
\end{remark}

\mysection{Applications of Main Results}        \label{sec:app}
\subsection{Scalar case}
In the scalar case (i.e., $m=1$), both conditions \eqref{LB} and \eqref{IH} are satisfied with $R_{max}=\infty$ and $N_0=N_0(n,\nu)$  in any domain $\Omega$; see e.g., \cite{GT}.
Also, in the scalar case, the Green's matrices are nonnegative scalar functions; see \cite{GW}.
Therefore, the following corollary is an immediate consequence of Theorem~\ref{thm1b}.

\begin{corollary}           \label{cor1}
If $m=1$, then for any domain $\Omega \subset \bR^n$, the Green's function $G(x,y)$ of $L$ in $\Omega$ exists and satisfies
\begin{equation}            \label{eq3.11e}
G(x,y) \leq C \abs{x-y}^{2-n},\quad \forall x,y\in\Omega,\quad x\neq y,
\end{equation}
where $C=C(n,\nu)$ is a universal constant independent of $\Omega$.
\end{corollary}

\begin{remark}
Corollary~\ref{cor1} is widely known (see e.g., \cite{GW, LSW}). However, it should be mentioned that unlike \cite{GW, LSW}, we do not need to assume that $\Omega$ is bounded.
\end{remark}

Also, in the scalar case, the condition \eqref{LH} is satisfied if $\Omega$ satisfies the condition \eqref{S}, the definition of which is given below.
In fact, if $L$ is a small $L^\infty$-perturbation of a diagonal system, then the condition \eqref{LH} is satisfied whenever $\Omega$ satisfies the condition \eqref{S}; see \S\ref{sec:AD} below.
\begin{CS}
There exist $\theta>0$ and $R_a\in (0,\infty]$ such that
\begin{equation}    \tag{S}\label{S}
\abs{B_R(x)\setminus\Omega}\ge \theta\abs{B_R(x)},\quad\forall x \in\partial\Omega,\quad \forall R\in (0, R_a).
\end{equation}
\end{CS}

The following corollary is then an easy consequence of Theorem~\ref{thm2}.

\begin{corollary}           \label{cor2}
Assume $m=1$ and let $G(x,y)$ be the Green's function of $L$ in $\Omega$, where $\Omega$ is a domain satisfying the condition \eqref{S}.
Then, $G(x,y)$ satisfies the estimate \eqref{eq3.11e}.
Moreover, for all $x,y\in \Omega$ satisfying $0<\abs{x-y}<R_a$, we have
\begin{equation}            \label{eq3.18yz}
G(x,y) \leq C\bigset{d_x\wedge \abs{x-y}}^{\mu_0} \bigset{d_y\wedge \abs{x-y}}^{\mu_0} \abs{x-y}^{2-n-2\mu_0},
\end{equation}
where $C=C(n,\nu,\theta)$ and $\mu_0=\mu_0(n,\nu,\theta)$. If $R_a<\infty$ and $\Omega$ is bounded, then for all $x, y\in \Omega$ with $x\neq y$, we have the estimate \eqref{eq3.18yz} with $C=C(n,\nu,\theta, R_a/ \diam \Omega)$.
\end{corollary}

\begin{example}
$\Omega=\bR_{+}^n$ satisfies the conditions \eqref{S} with $\theta=1/2$ and $R_a=\infty$.  
Therefore, Corollary~\ref{cor2} implies that for all $x, y\in \bR^n_+$ with $x\neq y$, we have
\[
G(x,y)\leq C\bigset{d_x\wedge \abs{x-y}}^{\mu_0} \bigset{d_y\wedge \abs{x-y}}^{\mu_0} \abs{x-y}^{2-n-2\mu_0},
\]
where $C=C(n,\nu)$ and $\mu_0=\mu_0(n,\nu)$.
\end{example}

\subsection{$L^\infty$-perturbation of diagonal systems}        \label{sec:AD}
Let $a^{\alpha\beta}(x)$ be scalar functions satisfying
\begin{equation}            \label{eqP-07}
a^{\alpha\beta}(x)\xi_\beta\xi_\alpha\ge \nu_0\bigabs{\vec \xi}^2,\quad\forall\xi\in\bR^n;\qquad \sum_{\alpha,\beta=1}^n \bigabs{a^{\alpha\beta}(x)}^2\le \nu_0^{-2},
\end{equation}
for all $x\in\bR^n$ with some constant $\nu_0\in (0,1]$.
Assume that $\Omega$ satisfies the condition \eqref{S} and let $A^{\alpha\beta}_{ij}(x)$ be the coefficients of the operator $L$.
We denote
\begin{equation}					\label{eqP-08w}
\mathscr{E}= \sup_{x\in \bR^{n}}\,\left\{\sum_{i,j=1}^m \sum_{\alpha,\beta =1}^n\, \Bigabs{A^{\alpha\beta}_{ij}(x)-a^{\alpha\beta}(x)\delta_{ij}}^2\right\}^{1/2},
\end{equation}
where $\delta_{ij}$ is the usual Kronecker delta symbol.
By \cite[Lemma 4.6]{HK07}, there exists a number $\mathscr{E}_0=\mathscr{E}_0(n,\nu_0,\theta)$ such that if $\mathscr{E}<\mathscr{E}_0$, then the condition \eqref{LH} is satisfied by $L$ in $\Omega$ with parameters $\mu_0=\mu_0(n,\nu_0, \theta)$, $N_1=N_1(n,m,\nu_0, \theta)$, and $R_{max}=R_a$.
Therefore, the following corollary is another easy consequence of Theorem~\ref{thm1b} and Theorem~\ref{thm2}.

\begin{corollary}           \label{cor2b}
Let $a^{\alpha\beta}(x)$ satisfy the condition \eqref{eqP-07}.
Assume that $\Omega$ satisfies the condition \eqref{S} and let $\mathscr{E}$ be defined as in \eqref{eqP-08w}, where $A^{\alpha\beta}_{ij}(x)$ are the coefficients of the operator $L$.
There exists $\mathscr{E}_0=\mathscr{E}_0(n,\nu_0,\theta)$ such that if $\mathscr{E} <\mathscr{E}_0$, then the Green's matrix $\vec G(x,y)$ of $L$ in $\Omega$ exists and for all $x,y\in \Omega$ satisfying $0<\abs{x-y}<R_a$, we have
\begin{equation}            \label{eq4.18y}
\abs{\vec G(x,y)} \leq C \bigset{d_x\wedge \abs{x-y}}^{\mu_0} \bigset{d_y\wedge \abs{x-y}}^{\mu_0} \abs{x-y}^{2-n-2\mu_0},
\end{equation}
where $C=C(n,m,\nu_0,\theta)$ and $\mu_0=\mu_0(n,\nu_0,\theta)$. If $R_a<\infty$ and $\Omega$ is bounded, then for all $x, y\in \Omega$ such that $x\neq y$, we have the estimate \eqref{eq4.18y} with $C=C(n,m, \nu_0,\theta, R_a/ \diam \Omega)$.
\end{corollary}

\begin{example}     \label{ex4.11}
Let $\Omega=\{ x\in \bR^n: x_n>\varphi(x')\}$, where $x=(x',x_n)$ and  $\varphi:\bR^{n-1}\to \bR$ is a Lipschitz function with the Lipschitz constant $K$. Then $\Omega$ satisfies the condition \eqref{S} with $\theta=\theta(n,K)$ and $R_a=\infty$.
If $L$ is a small $L^\infty$-perturbation of a diagonal system in the sense of Corollary~\ref{cor2b}, then the Green's matrix $\vec G(x,y)$ of $L$ in $\Omega$ exists and we have
\[
\abs{\vec G(x,y)} \leq C \bigset{d_x\wedge \abs{x-y}}^{\mu_0} \bigset{d_y\wedge \abs{x-y}}^{\mu_0} \abs{x-y}^{2-n-2\mu_0},\quad \forall x, y \in\Omega,\quad x\neq y, 
\]
where $C=C(n,m,\nu_0,K)$ and $\mu_0=\mu_0(n,\nu_0,K)$.
\end{example}

\subsection{Systems with VMO coefficients}      \label{sec:VMO}
For a measurable function $f$ on $\bR^n$, we set
\[
\omega_\delta(f):=\sup_{x\in\bR^n}\,\sup_{r \leq \delta}\, \fint_{B_r(x)} \bigabs{f(y)-\bar{f}_{x,r}}\,dy, \quad\forall \delta>0;\quad \bar{f}_{x,r}=\fint_{B_r(x)}f.
\]
We say that $f$ belongs to VMO if $\lim_{\delta\to 0} \omega_\delta(f)=0$; see \cite{Sarason}.

If the coefficients $\vec A^{\alpha\beta}$ of the operator $L$ are functions in VMO satisfying \eqref{eqP-02}, \eqref{eqP-03} and if $\Omega$ is a bounded $C^1$ domain, then the condition \eqref{LH} is satisfied with parameters $\mu_0$, $N_1$, and $R_{max}$ depending on $\Omega$ and $\omega_\delta(\vec A^{\alpha\beta})$ as well as on $n, m, \nu$. Therefore, we have the following corollary of Theorem~\ref{thm2}.

\begin{corollary}           \label{cor3}
 Let $\Omega$ be a bounded $C^1$ domain. 
 Suppose the coefficients $\vec A^{\alpha\beta}$ of the operator $L$ belong to VMO and satisfy the conditions \eqref{eqP-02}, \eqref{eqP-03}.
Then for all $x, y\in \Omega$ with $x\neq y$, we have
\[
\abs{\vec G(x,y)}  \leq C \bigset{d_x\wedge \abs{x-y}}^{\mu_0} \bigset{d_y\wedge \abs{x-y}}^{\mu_0} \abs{x-y}^{2-n-2\mu_0},
\]
where $C$ and $\mu_0$ are constants depending on $n,m,\nu, \Omega$, and  $\omega_\delta(\vec A^{\alpha\beta})$.
\end{corollary}

In the above corollary, one may assume that $\vec A^{\alpha\beta}$ satisfy the weaker Legendre-Hadamard condition and may even include lower order terms in the operator. More precisely, let
\begin{equation}            \label{eq4.15ip}
L_\lambda \vec u = -D_\alpha (\vec{A}^{\alpha\beta} D_\beta \vec u)+D_\alpha(\vec{B}^\alpha \vec u)+\hat{\vec{B}}{}^\alpha D_\alpha \vec u+ \vec{C} \vec u + \lambda \vec u,
\end{equation}
where $\vec A^{\alpha\beta}, \vec B^\alpha, \hat{\vec B}{}^\alpha$, and $\vec C$ are $m\times m$ matrix valued functions on $\bR^n$ satisfying
\begin{equation}            \label{eq4.16ys}
\left\{\;
\begin{aligned}
A^{\alpha\beta}_{ij}(x)\xi^j \xi^i \eta_\beta \eta_\alpha \ge \nu \bigabs{\vec \xi}^2  \bigabs{\vec \eta}^2,\quad \forall\vec\xi\in \bR^m,\,\,\,\forall \vec\eta\in\bR^n, \,\,\,\forall x\in\bR^n;\\
\sum_{\alpha, \beta=1}^n\bignorm{\vec A^{\alpha\beta}}_{L^\infty}^2 \leq \nu^{-2};\quad \sum_{\alpha=1}^n \left(\bignorm{\vec B^\alpha}_{L^\infty}^2 +\bignorm{\hat{\vec B}{}^\alpha}_{L^\infty}^2 \right)+\bignorm{\vec C}_{L^\infty}^2\leq \nu^{-2},
\end{aligned}
\right.
\end{equation}
for some constant $\nu\in (0,1]$, and $\lambda$ is a scalar constant.
\begin{corollary}           \label{cor4.15}
Let $\Omega$ be a bounded $C^1$ domain and let the operator $L_\lambda$ be defined as in \eqref{eq4.15ip} with the coefficients satisfying the conditions \eqref{eq4.16ys}.
We assume further that the leading coefficients $\vec A^{\alpha\beta}$ belong to VMO.
There exists $\lambda_0\geq 0$ such that if $\lambda> \lambda_0$, then the Green's matrix $\vec G(x,y)$ of $L_\lambda$ in $\Omega$ exists and for all $x, y\in\Omega$ with $x\neq y$, we have
\[
\abs{\vec G(x,y)}  \leq C \bigset{d_x\wedge \abs{x-y}}^{\mu_0} \bigset{d_y\wedge \abs{x-y}}^{\mu_0} \abs{x-y}^{2-n-2\mu_0},
\]
 where the constants $\mu_0$ and $C$ depend on $n,m,\nu, \Omega, \lambda$, and  $\omega_\delta(\vec A^{\alpha\beta})$.
\end{corollary}

To give a sketch of proof for Corollary~\ref{cor4.15}, first we note that for sufficiently large $\lambda$, we have the solvability of the following problem in $Y^{1,2}_0(\Omega)^m=W^{1,2}_0(\Omega)^m$:
\[
\left\{\begin{array}{r cl}
L_\lambda \vec u=\vec f &\text{in}& \Omega,\\
\vec u= 0&\text{on} &\partial\Omega,\end{array}\right.
\]
where $\vec f\in L^\infty_c(\Omega)$.
In particular, one can construct the ``averaged Green's matrix'' $\vec G^\rho(x,y)$ of $L_\lambda$ in $\Omega$ by following the argument in \cite[\S 4]{HK07}.
Also, it is not hard to see that the condition \eqref{eqP-09} in Remark~\ref{rmk:P-03} is satisfied in this case.
In particular, we have the condition \eqref{LB}. 
We point out that these are all the ingredients needed for construction of the Green's matrix $\vec G(x,y)$ of $L_\lambda$ in $\Omega$.
Then by modifying the proofs of Theorem~\ref{thm1b} and Theorem~\ref{thm2}, one can prove the above corollary; see Remark~\ref{rmk:P-04}. The details are left the the reader.

\begin{remark}
In Corollary \ref{cor3} and Corollary \ref{cor4.15}, the conditions of $\Omega$ and $\vec A^{\alpha\beta}$ can be relaxed. We may assume that $\Omega$ is a bounded Lipschitz domain with a sufficiently small Lipschitz constant, and $\omega_\delta(\vec A^{\alpha\beta})$ is also sufficiently small for some $\delta>0$;  see e.g., \cite{AT}.
\end{remark}

\subsection{Stationary Stokes system}
Let $\Omega\subset \bR^3$ be a bounded Lipschitz domain with connected boundary.
We consider the stationary Stokes system
\begin{equation}    \label{stokes-100}
-\Delta \vec u+\nabla p=0,\quad \dv \vec u=0 \quad\text{in }\,\,\Omega.
\end{equation}
It is known that the condition \eqref{LH} is satisfied in this setting; see \cite{Shen} and also  \cite{CC09}.
We also note that Caccioppoli's inequalities are available for the system \eqref{stokes-100}.
Then again, by modifying the proof of Theorem~\ref{thm2}, one can prove the following corollary.
\begin{corollary}
Let $\Omega\subset\mathbb{R}^3$ be a bounded Lipschitz domain with connected boundary. Let $\vec G$ be the Green's matrix of the stationary Stokes system \eqref{stokes-100} in $\Omega$. Then for all $x,y\in \Omega$ with $x \neq y$, we have
\begin{equation}
\label{eq4.17ue}
\abs{\vec G(x,y)}\leq C \bigset{d_x\wedge \abs{x-y}}^{\mu_0} \bigset{d_y\wedge \abs{x-y}}^{\mu_0} \abs{x-y}^{-1-2\mu_0}
\end{equation}
for some positive constants $C$ and $\mu_0$ depending on $\Omega$.
\end{corollary}

We remark that estimate \eqref{eq4.17ue} of the Green's matrices for the Lam\'e system and the Stokes system were recently shown in \cite{CC09} by a different method to ours.

\mysection{Proofs of Main Theorems}     \label{sec:p}
\subsection{Proof of Theorem \ref{thm1}}
Let $R\in (0,R_{max})$ and $y\in \Omega$ be arbitrary, but fixed.
Assume that $\vec f \in L^\infty(\Omega)$ is supported in $\Omega_R(y)$ and let $\vec u$ be defined by \eqref{eq2.9x}.
Notice that we may take $\vec u$ in place of $\vec \phi$ in \eqref{eq2.7i}.
Then by \eqref{eqP-14} we have
\begin{equation}        \label{eq3.1a}
\norm{\vec u}_{L^{2n/(n-2)}(\Omega)} \le C\norm{D\vec u}_{L^2(\Omega)}\le C \norm{\vec f}_{L^{2n/(n+2)}(\Omega)}\leq C  R^{1+n/2} \norm{\vec f}_{L^\infty(\Omega_R(y))}.
\end{equation}
Also, notice from Remark~\ref{rmk2.4q} that $\vec u \in W^{1,2}(\Omega_R(y))$.
Therefore, $\vec u$ is a weak solution of $\Lt \vec u=\vec f$ in $\Omega_R(y)$ vanishing on $\varSigma_R(y)$ and thus, by the condition \eqref{LB} we have
\begin{equation}        \label{eq3.2b}
\norm{\vec u}_{L^\infty(\Omega_{R/2}(y))} \le N_0 \left(R^{-n/2} \norm{\vec u}_{L^2(\Omega_R(y))}+ R^2 \norm{\vec f}_{L^\infty(\Omega_R(y))}\right).
\end{equation}
Then by \eqref{eq3.2b}, \eqref{eq3.1a}, and H\"older's inequality, we derive
\begin{equation}    \label{eq3.2v}
\norm{\vec u}_{L^\infty(\Omega_{R/2}(y))} \le C R^2 \norm{\vec f}_{L^\infty(\Omega_R(y))}.
\end{equation}
Hence, by \eqref{eq2.9x} and \eqref{eq3.2v}, we conclude that
\begin{equation}        \label{eq3.3y}
\Biggabs{\int_{\Omega_R(y)} \vec G(\cdot,y) \vec f \,} \le CR^2 \norm{\vec f}_{L^\infty(\Omega_R(y))},\quad \forall \vec f\in L^\infty(\Omega_R(y)).
\end{equation}
Therefore, by duality, we conclude from \eqref{eq3.3y} that
\begin{equation}    \label{eq3.3w}
\norm{\vec G(\cdot,y)}_{L^1(\Omega_R(y))}\le C R^2.
\end{equation}

Next, notice that  \eqref{LB} implies that for $x\in\Omega$ and $R \in(0,R_{max})$, if $\vec v \in W^{1,2}(\Omega_R(x))$ is a weak solution of $L \vec v=0$ in $\Omega_R=\Omega_R(x)$ vanishing  on $\varSigma_R(x)$, then we have
\[
\norm{\vec v}_{L^\infty(\Omega_{R/2})}  \le N_0 R^{-n/2} \norm{\vec v}_{L^2(\Omega_R)}.
\]
Then, by a standard argument (see e.g., \cite[pp. 80--82]{Gi93}) we also have
\begin{equation}    \label{eq2.8r}
\norm{\vec v}_{L^\infty(\Omega_{R/2})}  \le C_p R^{-n/p} \norm{\vec v}_{L^p(\Omega_R)}, \quad\forall p>0,
\end{equation}
where the constant $C_p$ depends on $n$, $N_0$, and $p$.

Now, for any $x\in\Omega$ such that $0<\abs{x-y}<R_{max}/2$, set $R:=2\abs{x-y}/3$.
Notice that Definition~\ref{def2} implies that $\vec G(\cdot,y)\in W^{1,2}(\Omega_R(x))$ and satisfies $L\vec G(\cdot,y)= 0$ weakly in $\Omega_R(x)$ and $\vec G(\cdot,y)=0$ on $\varSigma_R(x)$.
Therefore, by \eqref{eq2.8r} and \eqref{eq3.3w}, we have
\begin{equation}        \label{eq4.6m}
\abs{\vec G(x,y)} \le C R^{-n} \norm{\vec G(\cdot,y)}_{L^1(\Omega_R(x))} \le C R^{-n} \norm{\vec G(\cdot,y)}_{L^1(\Omega_{3R}(y))} \le C R^{2-n}.
\end{equation}
We have thus shown that
\[
\abs{\vec G(x,y)} \leq  C \abs{x-y}^{2-n}\quad \text{for all $x,y\in\Omega$ satisfying $0<\abs{x-y}<R_{max}/2$},
\]
where $C=C(n,m,\nu, N_0)$.
The theorem then follows from Remark~\ref{rmk3.3b}.
\hfill\qedsymbol

\subsection{Proof of Theorem~\ref{thm1b}}
The existence of the Green's matrices $\vec G(x,y)$ and ${}^t\vec G(x,y)$ as well as the identity \eqref{eq2.20f} is a consequence of the condition \eqref{IH}; see \cite[Theorem~4.1]{HK07} and \cite[Eq. (4.34)]{HK07}.
Then, by Theorem~\ref{thm1}, the condition \eqref{LB} yields the estimate \eqref{eq2.17d}.
Also, by \cite[Eq. (4.24)]{HK07}, we find that the estimate in \eqref{eq3.9c} is valid for $0<r<(d_y\wedge R_c)/2$. To give a full proof of \eqref{eq3.9c}, we need to make use of the estimate \eqref{eq2.17d} and adapt the arguments used in \cite{HK07} as follows.

For $\rho>0$, let $\vec G^\rho(\cdot,y)\in Y^{1,2}_0(\Omega)$ be the averaged Green's matrix of $L$ in $\Omega$ as constructed in \cite[\S 4.1]{HK07}. Notice that by \cite[Eq. (4.3)]{HK07}, we have
\begin{equation}
\label{eqG-04}
\int_{\Omega}A^{\alpha\beta}_{ij} D_\beta G^\rho_{jk}(\cdot,y)D_\alpha u^i = \fint_{\Omega_\rho(y)}u^k, \quad\forall\vec u\in Y^{1,2}_0(\Omega).
\end{equation}
Also, by \cite[Eq. (4.2)]{HK07}, we have
\begin{equation}
\label{eqG-02}
\norm{D\vec G^\rho(\cdot,y)}_{L^2(\Omega)}\le C \abs{\Omega_\rho(y)}^{(2-n)/2n} \leq C \rho^{(2-n)/2},
\end{equation}
where $C=C(n,m,\nu)$.
Denote by $\bH$ the Hilbert space $Y^{1,2}_0(\Omega)^m$ with the inner product \[\ip{\vec u, \vec v}:=\int_\Omega D_{\alpha} u^i D_{\alpha}v^i.\]
For all $\vec f\in L^\infty_c(\Omega)$, the linear functional \[\vec w\mapsto \int_\Omega \vec f\cdot \vec w\] is bounded on $\bH$. Hence, by the Lax-Milgram lemma there is a unique $\vec u\in \bH$ satisfying
\begin{equation}
\label{eqE-11}
\int_\Omega A^{\alpha\beta}_{ij} D_\beta w^j D_\alpha u^i= \int_\Omega \vec f\cdot \vec w,
\quad\forall\vec w\in \bH.
\end{equation}
Thus, if we set $\vec w$ to be the $k$-th column of $\vec G^\rho(\cdot,y)$ in \eqref{eqE-11}, we obtain from \eqref{eqG-04} that
\begin{equation}
\label{eqE-12}
\int_{\Omega} G^\rho_{ik}(\cdot,y) f^i=\fint_{\Omega_\rho(y)}u^k.
\end{equation}
Also, if we set $\vec w=\vec u$ in \eqref{eqE-11}, it follows from \eqref{eqP-14} that
\begin{equation}
\label{eqE-13}
\norm{\vec u}_{L^{2n/(n-2)}(\Omega)} \le C\norm{D\vec u}_{L^2(\Omega)}\le C \norm{\vec f}_{L^{2n/(n+2)}(\Omega)}.
\end{equation}

Let us now assume that $\vec f$ is supported in $\Omega_R:=\Omega_R(y)$, where $y\in \Omega$ and $R \in (0,R_{max})$ are arbitrary, but fixed.
Notice that $\vec u \in W^{1,2}(\Omega_R)$ and $\vec u$ is a weak solution of $\Lt \vec u =\vec f$ in $\Omega_R$ vanishing on $\varSigma_R$.
Therefore, the condition \eqref{LB} implies that
\[
\norm{\vec u}_{L^\infty(\Omega_{R/2})} \le N_0 \left(R^{-n/2} \norm{\vec u}_{L^2(\Omega_R)}+ R^{2} \norm{\vec f}_{L^\infty(\Omega_R)}\right).
\]
On the other hand, \eqref{eqE-13} and H\"older's inequality yields
\[
\norm{\vec u}_{L^2(\Omega_R)} \leq C R^{2+n/2} \norm{\vec f}_{L^\infty(\Omega_R)}
\]
By combining the above two inequalities, we obtain
\begin{equation}    \label{eq3.6a}
\norm{\vec u}_{L^\infty(\Omega_{R/2})} \le C R^{2} \norm{\vec f}_{L^\infty(\Omega_R)}.
\end{equation}
Then, by \eqref{eqE-12} and \eqref{eq3.6a} we derive
\[
\Biggabs{\int_{\Omega_R} G^\rho_{ik}(\cdot,y) f^i} \le C R^{2} \norm{\vec f}_{L^\infty(\Omega_R)},\quad \forall \vec f\in L^\infty(\Omega_R),\quad \forall \rho \in (0,R/2).
\]
Therefore, by duality, we conclude that
\[
\norm{\vec G^\rho(\cdot,y)}_{L^1(\Omega_R(y))}\le C R^{2},\quad \forall \rho \in(0,R/2).
\]
Now, for any $x\in\Omega$ such that $0<\abs{x-y}<R_{max}/2$, let us take $R:=2\abs{x-y}/3$.
Notice that if $\rho<R/2$, then $\vec G^\rho(\cdot,y)\in W^{1,2}(\Omega_R(x))$ and satisfies $L\vec  G^\rho(\cdot,y)= 0$ in $\Omega_R(x)$ and vanishes on $\varSigma_R(x)$.
Therefore, as in \eqref{eq4.6m}, we have
\[
\abs{\vec G^\rho(x,y)}  \le C R^{-n} \norm{\vec G^\rho(\cdot,y)}_{L^1(\Omega_{3R}(y))}\le CR^{2-n}.
\]
We have thus proved that for any $x,y\in \Omega$ satisfying $0<\abs{x-y}<R_{max}/2$, we have
\begin{equation}        \label{eq3.8a}
\abs{\vec G^\rho(x,y)} \le C \abs{x-y}^{2-n},\quad \forall \rho <\abs{x-y}/3,
\end{equation}
where  $C=C(n,m,\nu, N_0)$.

Next, fix any $r\in (0,R_{max}/2)$ and let $\vec v_\rho$ be the $k$-th column of $\vec G^\rho(\cdot,y)$, where $k=1,\ldots, m$ and $0<\rho<r/6$. Let $\eta$ be a smooth function on $\bR^n$ satisfying
\begin{equation}        \label{eq4.19h}
0\leq \eta\leq 1,\quad \eta\equiv 1\,\text{ on }\,\bR^n\setminus B_r(y),\quad \eta\equiv 0\,\text{ on }\, B_{r/2}(y),\quad\text{and}\quad \abs{D\eta} \le 4/r.
\end{equation}
We set $\vec u=\eta^2\vec v_\rho$ in \eqref{eqG-04} and then use \eqref{eq3.8a} to obtain
 \begin{equation}       \label{eq4.24y}
\int_\Omega \eta^2 \bigabs{D\vec v_\rho}^2 \le C \int_\Omega \bigabs{D\eta}^2 \bigabs{\vec v_\rho}^2 \le C r^{-2} \int_{B_r(y)\setminus B_{r/2}(y)} \abs{x-y}^{2(2-n)}\,dx\leq  C r^{2-n}.
\end{equation}
Therefore, by \eqref{eq4.19h}, \eqref{eqP-14}, and \eqref{eq4.24y},  we obtain
\[
\bignorm{\vec v_\rho}_{L^{2n/(n-2)}(\Omega\setminus B_r(y))} \leq \bignorm{\eta \vec v_\rho}_{L^{2n/(n-2)}(\Omega)}\leq C \bignorm{D \big(\eta \vec v_\rho\big)}_{L^2(\Omega)} \le C r^{(2-n)/2}
\]
provided that $0<\rho<r/6$. On the other hand, if $\rho\ge r/6$, then \eqref{eqG-02} implies
\[
\bignorm{\vec v_\rho}_{L^{2n/(n-2)}(\Omega\setminus B_{r}(y))} \leq \bignorm{\vec v_\rho}_{L^{2n/(n-2)}(\Omega)} \leq C \bignorm{D \vec v_\rho}_{L^2(\Omega)} \leq C r^{(2-n)/2}.
\]
By combining the above two estimates, we obtain
\begin{equation}							\label{eqG-20}
\norm{\vec G^\rho(\cdot,y)}_{L^{2n/(n-2)}(\Omega\setminus B_r(y))} \le C r^{(2-n)/2}, \quad \forall r \in (0,R_{max}/2),\quad \forall \rho>0.
\end{equation}
Notice from \eqref{eq4.24y} and \eqref{eq4.19h} that for $0<\rho<r/6$, we have
\[
\norm{D\vec G^\rho(\cdot,y)}_{L^2(\Omega\setminus B_r(y))}\leq C r^{(2-n)/2}.
\]
In the case when $\rho \ge r/6$, we obtain from \eqref{eqG-02} that
\[
\norm{D\vec G^\rho(\cdot,y)}_{L^2(\Omega\setminus B_r(y))} \le \norm{D\vec G^\rho(\cdot,y)}_{L^2(\Omega)} \le C \rho^{(2-n)/2} \le C r^{(2-n)/2}.
\]
By combining the above two inequalities, we obtain
\begin{equation}							\label{eqG-14}
\norm{D\vec G^\rho(\cdot,y)}_{L^2(\Omega\setminus B_r(y))}\leq C r^{(2-n)/2},\quad \forall r\in(0,R_{max}/2),\quad \forall \rho>0.
\end{equation}

Notice from \cite[Eq. (4.19)]{HK07} that there exists a sequence $\{\rho_\mu\}_{\mu=1}^\infty$ tending to zero such that $\vec G^{\rho_\mu}(\cdot,y) \rightharpoonup \vec G(\cdot,y)$ weakly in $Y^{1,2}_0(\Omega\setminus B_r(y))$ for all $r>0$.
Therefore, \eqref{eq3.9c} follows from \eqref{eqG-20}, \eqref{eqG-14}, and the obvious fact that $R_{max}/2$ and $R_{max}$ are comparable to each other in the case when $R_{max}<\infty$.
The proof is complete.
\hfill\qedsymbol

\subsection{Proof of Theorem~\ref{thm1c}}
As we mentioned in the proof of Theorem~\ref{thm1b}, the condition \eqref{IH} implies the existence of the Green's matrices $\vec G(x,y)$ and ${}^t \vec G(x,y)$ in $\Omega$ and also the identity \eqref{eq2.20f}.
Hence, if $\vec G(x,y)$ satisfies the estimate \eqref{eq2.17dd}, then so does ${}^t \vec G(x,y)$.
Therefore, by the symmetry, it is enough to prove \eqref{LB} for weak solutions of $\Lt \vec u = \vec f$.

Let $x\in\Omega$ and $R\in (0,R_{max})$ be given.	
Assume that $\vec u \in W^{1,2}(\Omega_R(x))$ is a weak solution of $\Lt\vec u=\vec f$ in $\Omega_R(x)$ vanishing on $\varSigma_R(x)$, where $\vec f\in L^\infty(\Omega_R(x))$.
Then, we have
\begin{equation}                \label{eq4.23h}
\int_{\Omega_R(x)} A^{\alpha\beta}_{ij}D_\alpha u^i D_\beta w^j=\int_{\Omega_R(x)}f^j w^j,\quad \forall \vec w \in W^{1,2}_0(\Omega_R(x)).
\end{equation}

Let $\vec G^\rho(\cdot,y)$ be the averaged Green's matrix of $L$ in $\Omega$ as in the proof of Theorem~\ref{thm1b}.
Set $\vec v=\zeta \vec u$, where $\zeta$ is a smooth cut-off function on $\bR^n$ satisfying
\begin{equation}                \label{eq4zeta}
0\leq \zeta\leq 1,\quad \supp \zeta \subset B_{R/2}(x),\quad \zeta\equiv 1\,\text{ on }\, B_{3R/8}(x),\quad\text{and}\quad \abs{D\zeta} \le 16/R.
\end{equation}
Notice that $\vec v \in Y^{1,2}_0(\Omega)$ and thus, by \eqref{eqG-04}, we obtain
\begin{equation}                \label{eq4.24e}
\fint_{\Omega_\rho(y)} \zeta u^k= \int_{\Omega}A^{\alpha\beta}_{ij} D_\beta G^\rho_{jk}(\cdot,y)u^iD_\alpha \zeta + \int_{\Omega}A^{\alpha\beta}_{ij} D_\beta G^\rho_{jk}(\cdot,y) \zeta D_\alpha u^i.
\end{equation}

On the other hand, notice that $\zeta \vec G^\rho(\cdot,y) \in W^{1,2}_0(\Omega_R(x))$. Hence, if we set $\vec w$ to be the $k$-th column of $\zeta\vec G^\rho(\cdot,y)$, then by \eqref{eq4.23h}, we obtain
\begin{equation}                \label{eq4.25f}
\int_{\Omega_R(x)} A^{\alpha\beta}_{ij}D_\alpha u^i G^\rho_{jk}(\cdot,y)D_\beta\zeta+\int_{\Omega_R(x)} A^{\alpha\beta}_{ij}D_\alpha u^i \zeta D_\beta G^\rho_{jk}(\cdot,y)=\int_{\Omega_R(x)}\zeta f^j G^\rho_{jk}(\cdot,y).
\end{equation}
Recall that $\supp \zeta \subset B_{R/2}(x)$.
Therefore, by combining \eqref{eq4.24e} and \eqref{eq4.25f}, we obtain
\begin{align}                   \label{eq4.26u}
\fint_{\Omega_\rho(y)} \zeta u^k&= \int_\Omega A^{\alpha\beta}_{ij} D_\beta G^\rho_{jk}(\cdot,y)u^iD_\alpha \zeta - \int_\Omega A^{\alpha\beta}_{ij}G^\rho_{jk}(\cdot,y)D_\alpha u^i D_\beta\zeta +\int_\Omega \zeta f^j G^\rho_{jk}(\cdot,y)\\
                            \nonumber
&=: I_1 + I_2 + I_3.
\end{align}

Now, assume that $y\in \Omega_{R/4}(x)$. Notice from \eqref{eq4zeta} that we have $\dist(y, \supp D\zeta)> R/8$. Set $r=R/8\wedge (d_y\wedge R_c)$.
By \cite[Eq. (4.17)]{HK07} and \cite[Eq. (4.19)]{HK07}, there exists a sequence $\{\rho_\mu\}_{\mu=1}^\infty$ tending to zero such that $\vec G^{\rho_\mu}(\cdot,y)\rightharpoonup \vec G(\cdot,y)$ weakly in $W^{1,q}(B_r(y))$ for $q\in (1,\frac{n}{n-1})$ and $\vec G^{\rho_\mu}(\cdot,y) \rightharpoonup \vec G(\cdot,y)$ weakly in $Y^{1,2}_0(\Omega\setminus B_r(y))$.
Notice that
\begin{align*}
I_1+I_2 &= \int_{\Omega\setminus B_r(y)} A^{\alpha\beta}_{ij} D_\beta G^\rho_{jk}(\cdot,y)u^iD_\alpha \zeta - \int_{\Omega\setminus B_r(y)} A^{\alpha\beta}_{ij}G^\rho_{jk}(\cdot,y)D_\alpha u^i D_\beta\zeta;\\
I_3 &=\int_{\Omega\setminus B_r(y)} \zeta f^j G^\rho_{jk}(\cdot,y)+\int_{B_r(y)} \zeta f^j G^\rho_{jk}(\cdot,y).
\end{align*}
Therefore, by taking limits in \eqref{eq4.26u} and using \eqref{eq4zeta}, we have for almost all $y\in \Omega_{R/4}(x)$,
\begin{align}               \label{eq4.28v}
u^k(y)&= \int_\Omega A^{\alpha\beta}_{ij} D_\beta G_{jk}(\cdot,y)u^iD_\alpha \zeta - \int_\Omega A^{\alpha\beta}_{ij}G_{jk}(\cdot,y)D_\alpha u^i D_\beta\zeta +\int_\Omega \zeta f^j G_{jk}(\cdot,y)\\
                        \nonumber
&=: I_1'+ I_2'+ I_3'.
\end{align}
Denote $A_R(y)=\Omega_{3R/4}(y)\setminus B_{R/8}(y)$.
By using H\"older's inequality and \eqref{eq4zeta} we obtain
\begin{align*}
\bigabs{I_1'} &\leq C R^{-1}\norm{D\vec G(\cdot,y)}_{L^2(A_R(y))}\, \norm{\vec u}_{L^2(\Omega_{R/2}(x))},\\
\bigabs{I_2'} &\leq C R^{-1} \norm{\vec G(\cdot,y)}_{L^2(A_R(y))}\, \norm{D\vec u}_{L^2(\Omega_{R/2}(x))},\\
\intertext{and}
\bigabs{I_3'} &\leq C \norm{\vec G(\cdot,y)}_{L^1(\Omega_{3R/4}(y))}\, \norm{\vec f}_{L^\infty(\Omega_{R/2}(x))}.
\end{align*}
Denote $\tilde A_R(y)=\Omega_R(y)\setminus B_{R/16}(y)$.
Observe that $\vec G(\cdot,y) \in W^{1,2}(\tilde A_R(y))$ and it satisfies $L\vec G(\cdot,y)=0$ weakly in $\tilde A_R(y)$ and vanishes on $\partial\Omega \cap \partial \tilde A_R(y)$.
Then by the Caccioppoli's inequality and the estimate \eqref{eq2.17dd}, we obtain
\[
\norm{D\vec G(\cdot,y)}_{L^2(A_R(y))} \leq C R^{-1} \norm{\vec G(\cdot,y)}_{L^2(\tilde A_R(y))} \le C R^{(2-n)/2}.
\]
Therefore, we have
\begin{equation}							\label{eq:ip}
\bigabs{I_1'} \leq C R^{-n/2}\norm{\vec u}_{L^2(\Omega_R(x))}.
\end{equation}
By setting $\vec w= \eta^2 \vec u$ in \eqref{eq4.23h}, where $\eta \in C^\infty_c(B_R(x))$ is a cut-off function such that $\eta\equiv 1$ on $B_{R/2}(x)$ and $\abs{D\eta} \leq 4/R$, and use a standard argument, we derive
\[
\int_{\Omega_R(x)} \eta^2 \abs{D\vec u}^2 \leq C \int_{\Omega_R(x)} \abs{D\eta}^2 \abs{\vec u}^2+C\int_{\Omega_R(x)} \abs{\eta \vec f} \abs{\eta \vec u}.
\]
By the Sobolev inequality, H\"older's inequality, and Cauchy's inequality, we obtain
\begin{align*}
\int_{\Omega_R(x)} \abs{\eta\vec f} \abs{\eta \vec u}
&\leq \frac{\epsilon}{2}\,\norm{D(\eta \vec u)}_{L^2(\Omega_R(x))}^2 +C \epsilon^{-1}\norm{\eta \vec f}_{L^{2n/(n+2)}(\Omega_R(x))}^2\\
&\leq \epsilon \int_{\Omega_R(x)} \abs{D\eta}^2 \abs{\vec u}^2 +\epsilon \int_{\Omega_R(x)} \eta^2 \abs{D\vec u}^2 + C\epsilon^{-1}R^{n+2}\norm{\vec f}_{L^\infty(\Omega_R(x))}^2.
\end{align*}
By choosing $\epsilon$ small enough, we then obtain
\[
\int_{\Omega_R(x)} \eta^2 \abs{D\vec u}^2 \leq C \int_{\Omega_R(x)} \abs{D\eta}^2 \abs{\vec u}^2+C R^{n+2}\norm{\vec f}_{L^\infty(\Omega_R(x))}^2.
\]
Therefore, by using the estimate \eqref{eq2.17dd} we derive
\begin{equation}            \label{eq:iip}
\bigabs{I_2'} \leq C R^{-n/2} \norm{\vec u}_{L^2(\Omega_R(x))}+C R^2 \norm{\vec f}_{L^\infty(\Omega_R(x))}.
\end{equation}
By using the estimate \eqref{eq2.17dd} again, we also obtain
\begin{equation}            \label{eq:iiip}
\bigabs{I_3'} \leq C R^2 \norm{\vec f}_{L^\infty(\Omega_R(x))}.
\end{equation}
By combining above estimates \eqref{eq:ip}, \eqref{eq:iip}, and \eqref{eq:iiip}, we conclude from \eqref{eq4.28v} that
\begin{equation}							\label{eq5.29tp}
\norm{\vec u}_{L^\infty(\Omega_{R/4}(x))} \leq C \left(R^{-n/2} \norm{\vec u}_{L^2(\Omega_R)}+ R^2 \norm{\vec f}_{L^\infty(\Omega_R)} \right),
\end{equation}
where $C=C(n,m,\nu,C_0)$.
Since \eqref{eq5.29tp} holds for all $x\in\Omega$ and $R\in(0,R_{max})$, we obtain \eqref{LB} by a standard covering argument.
\hfill\qedsymbol

\subsection{Proof of Theorem~\ref{thm2}}				\label{ss5.4}
Notice that by Lemma~\ref{lem2.19} and Theorem~\ref{thm1b}, we have
\begin{equation}            \label{eq2.25t}
\abs{\vec G(x,y)} \leq C_0 \abs{x-y}^{2-n}\quad \text{if }\, 0<\abs{x-y}<R_{max},
\end{equation}
where $C_0=C_0(n,m,\nu, \mu_0,N_1)$.
To prove the estimate \eqref{eq2.26s}, we first claim that
\begin{equation}			\label{eq5.30ad}
\abs{\vec G(x,y)}  \leq C \bigset{d_x\wedge \abs{x-y}}^{\mu_0} \abs{x-y}^{2-n-\mu_0}\quad \text{if }\, 0<\abs{x-y}<R_{max},
\end{equation}
where $C=C(n,m,\nu, \mu_0,N_1)$.
The following lemma is the key to prove \eqref{eq5.30ad}.
										
\begin{lemma}           \label{lem3.6}
Assume the condition \eqref{LH}.
For $R\in (0, R_{max})$ and $x\in\Omega$ such that $d_x < R/2$, let $\vec u \in W^{1,2}(\Omega_R(x))$ be a weak solution of $L\vec u=0$ in $\Omega_R(x)$ vanishing on $\varSigma_R(x)$.
Then, we have
\begin{equation}            \label{eq3.7m}
\abs{\vec u(x)} \le C d_x^{\mu_0} R^{1-n/2-\mu_0}\norm{D\vec u}_{L^2(\Omega_R(x))},
\end{equation}
where $C=C(n,m,\nu, \mu_0, N_1)$.
\end{lemma}
\begin{proof}
Let $\tilde{\vec u}$ be an extension of $\vec u$ by zero on $B_R(x)\setminus \Omega$.
Notice that $\tilde{\vec u}\in W^{1,2}(B_R(x))$ and $D\tilde{\vec u} = \chi_{\Omega_R} D\vec u$ in $B_R(x)$.
Then by the Poincar\'e's inequality and \eqref{LH}, we find that for all $r\in (0,R/2]$ and $y\in B_{R/2}(x)$, we have
\begin{align*}
\int_{B_r(y)} \Abs{\tilde{\vec u}-\tilde{\vec u}_r}^2 & \le C r^2 \int_{B_r(y)} \abs{D \tilde{\vec u}}^2 = C r^2 \int_{\Omega_r(y)} \abs{D \vec u}^2\\
& \le C r^{n+2\mu_0} R^{-n+2-2\mu_0} \norm{D\vec u}_{L^2(\Omega_R(x))}^2.
\end{align*}
Then by the Campanato's characterization of H\"older seminorms, we have
\begin{equation}            \label{eq3.8x}
[\tilde{\vec u}]_{C^{\mu_0}(B_{R/2}(x))} \le C R^{1-n/2-\mu_0} \norm{D\vec u}_{L^2(\Omega_R(x))}.
\end{equation}
For any $r\in (d_x, R/2)$, there is $x'\in B_{R/2}(x) \setminus \Omega$ such that $\abs{x-x'}=r$. By \eqref{eq3.8x} we obtain
\[
\abs{\vec u(x)}=\bigabs{\tilde{\vec u}(x)-\tilde{\vec u}(x')} \leq C r^{\mu_0} R^{1-n/2-\mu_0}\norm{D\vec u}_{L^2(\Omega_R(x))}.
\]
By taking limit $r\to d_x$ in the above inequality, we derive \eqref{eq3.7m}.
\end{proof}

Now we are ready to prove the claim \eqref{eq5.30ad}.
We may assume that $d_x <\abs{x-y}/4$ because otherwise \eqref{eq5.30ad} follows from \eqref{eq2.25t}.
We then set $R=\abs{x-y}/2$ and $\vec u$ to be $k$-th column of $\vec G(\cdot,y)$, for $k=1,\ldots, m$, in Lemma~\ref{lem3.6} to obtain
\[
\abs{\vec G(x,y)} \leq C d_x^{\mu_0}R^{1-n/2-\mu_0} \norm{D \vec G(\cdot,y)}_{L^2(\Omega_R(x))};\quad R=\abs{x-y}/2.
\]
On the other hand, since $\Omega_R(x)\subset \Omega\setminus B_R(y)$ and $R<R_{max}/2$, we have by \eqref{eq3.9c} that
\[
\norm{D\vec G(\cdot,y)}_{L^2(\Omega_R(x))}  \le \norm{D\vec G(\cdot,y)}_{L^2(\Omega\setminus B_R(y))}  \le C R^{(2-n)/2}.
\]
By combining the above two inequalities, we find that
\[
\abs{\vec G(x,y)} \leq C d_x^{\mu_0}\abs{x-y}^{2-n-\mu_0},
\]
which implies \eqref{eq5.30ad} since we assume $d_x<\abs{x-y}/4$. We have proved the claim.

We prove the estimate \eqref{eq2.26s} using \eqref{eq5.30ad}.
Since the condition \eqref{LH} is symmetric between $L$ and $\Lt$, by applying the above argument to ${}^t\vec G(x,y)$ and then interchanging $x$ and $y$, we obtain, via the identity \eqref{eq2.20f} and Remark~\ref{rmk3.12}, that
\begin{equation}				\label{eq5.35wb}
\abs{\vec G(x,y)} \leq C \bigset{d_y \wedge \abs{x-y}}^{\mu_0} \abs{x-y}^{2-n-\mu_0}\quad\text{if }\,0<\abs{x-y}<2R_{max}.
\end{equation}

Again, we may assume that $d_x <\abs{x-y}/8$ to prove \eqref{eq2.26s} because otherwise \eqref{eq2.26s} would follow from \eqref{eq5.35wb}.
We set $R=\abs{x-y}/4$ and $\vec u$ to be $k$-th column of $\vec G(\cdot,y)$, for $k=1,\ldots, m$, in Lemma~\ref{lem3.6}, and then use the Caccioppoli's inequality to obtain
\begin{equation}				\label{eq5.36ui}
\abs{\vec G(x,y)} \leq C d_x^{\mu_0}R^{1-n/2-\mu_0} \norm{D \vec G(\cdot,y)}_{L^2(\Omega_R(x))}\leq C d_x^{\mu_0} R^{-n/2-\mu_0} \norm{\vec G(\cdot,y)}_{L^2(\Omega_{2R}(x))}.
\end{equation}
Notice that for all $z\in \Omega_{2R}(x)$, we have $2R<\abs{z-y}<6R$. Therefore, by the assumption $R=\abs{x-y}/4$ and \eqref{eq5.35wb}, we obtain
\begin{equation}				\label{eq5.37ce}
\abs{\vec G(z,y)}  \leq  C \bigset{d_y \wedge \abs{x-y}}^{\mu_0} \abs{x-y}^{2-n-\mu_0},\quad \forall z\in \Omega_{2R}(x).
\end{equation}
By combining \eqref{eq5.36ui} and \eqref{eq5.37ce}, we obtain
\[
\abs{\vec G(x,y)} \leq C d_x^{\mu_0} \abs{x-y}^{-\mu_0}\bigset{d_y \wedge \abs{x-y}}^{\mu_0} \abs{x-y}^{2-n-\mu_0},
\]
which implies \eqref{eq2.26s} since we assume  $d_x <\abs{x-y}/8$.
This completes the proof of \eqref{eq2.26s} for all $x, y\in \Omega$ satisfying $0<\abs{x-y}<R_{max}$.

Next, we prove the second part of the theorem. Suppose $R_{max}<\infty$ and $\diam(\Omega)<\infty$.
Let $x, y$ be arbitrary but fixed points in $\Omega$ satisfying $|x-y|\ge R_{max}/2$.
Let $R=R_{max}/4$ and $\vec v$ be the $k$-th column of $\vec G(\cdot,y)$ for $k=1,\ldots,m$.
Notice that $\vec v \in W^{1,2}(\Omega_R(x))$ and $\vec v$ is a weak solution of $L\vec v$ =0 in $B_R(x)$ vanishing on $\varSigma_R(x)$. Hence, by \eqref{eq2.8r} with $p=2n/(n-2)$, we have
\[
\norm{\vec G(\cdot,y)}_{L^\infty(\Omega_{R/2}(x))}  \le C R^{(2-n)/2} \norm{\vec G(\cdot,y)}_{L^{2n/(n-2)}(\Omega_R(x))}.
\]
Therefore, by the above estimate and \eqref{eq3.9c} we have
\begin{equation}            \label{eq4.37qe}
\abs{\vec G(x,y)} \leq C R^{(2-n)/2}\norm{\vec G(\cdot,y)}_{Y^{1,2}(\Omega\setminus B_R(y))}  \le C R^{2-n} \leq C R_{max}^{2-n}.
\end{equation}

On the other hand,  if we set $R=R_{max}/4$ and $\vec u$ to be the $k$-th column of $\vec G(\cdot,y)$ for $k=1,\ldots,m$, in Lemma~\ref{lem3.6}, then by  \eqref{eq3.9c} again, we find that if $d_x< R/4=R_{max}/8$, then
\[
\abs{\vec G(x,y)} \leq C d_x^{\mu_0}R^{1-n/2-\mu_0} \norm{D \vec G(\cdot,y)}_{L^2(\Omega_R(x))}\le C d_x^{\mu_0} R_{max}^{2-n-\mu_0}.
\]

By combining \eqref{eq4.37qe} and the above estimate, we derive the following conclusion.
\begin{equation}            \label{eq4.38ty}
\abs{\vec G(x,y)} \leq C (d_x \wedge R_{max})^{\mu_0}R_{max}^{2-n-\mu_0}\quad \text{whenever }\, \abs{x-y} \geq R_{max}/2.
\end{equation}
Then, by using \eqref{eq4.38ty} and arguing similarly as above, we obtain
\begin{equation}						\label{eq5.42a}
\abs{\vec G(x,y)} \leq C (d_x \wedge R_{max})^{\mu_0} (d_y \wedge R_{max})^{\mu_0} R_{max}^{2-n-2\mu_0}\quad \text{whenever }\, \abs{x-y} \geq R_{max}.
\end{equation}
Therefore, we conclude from \eqref{eq5.42a} that for all $x,y \in\Omega$ satisfying $\abs{x-y} \geq R_{max}$, we have
\[
\abs{\vec G(x,y)}  \leq C \bigset{d_x\wedge \abs{x-y}}^{\mu_0} \bigset{d_y \wedge \abs{x-y}}^{\mu_0} (R_{max}/\diam\Omega)^{2-n-2\mu_0} \abs{x-y}^{2-n-2\mu_0}.
\]
From the above estimate, we obtain \eqref{eq2.26s} in case when $\abs{x-y} \geq R_{max}$, with the constant $C$ replaced by $(R_{max}/\diam\Omega)^{2-n-2\mu_0} C$.
Recall that we already have \eqref{eq2.26s} in the case when $0<\abs{x-y}<R_{max}$.
The proof is complete.
\hfill\qedsymbol

\mysection{Appendix}    \label{sec:a}

\begin{lemma}   \label{lem2.19}
Assume the condition \eqref{LH}.
For any $p\in (n/2,\infty]$, there exists a constant $C=C(n,m,\nu,\mu_0,N_1,p)$ such that for all $x\in\Omega$, $R\in (0,R_{max})$, and $\vec f \in L^p(\Omega_R(x))$, the following holds:
If $\vec u \in W^{1,2}(\Omega_R(x))$ is a weak solution of either $L \vec u=\vec f$ or $\Lt \vec u=\vec f$ in $\Omega_R(x)$ and vanishes on $\varSigma_R(x)$, then we have
\begin{equation}							\label{eq6.01ap}
\norm{\vec u}_{L^\infty(\Omega_{R/2})} \le C \left(R^{-n/2} \norm{\vec u}_{L^2(\Omega_R)}+ R^{2-n/p} \norm{\vec f}_{L^\infty(\Omega_R)} \right);\quad \Omega_R=\Omega_R(x).
\end{equation}
In particular, the condition \eqref{LB} holds with the same $R_{max}$ and $N_0=N_0(n,m,\nu,\mu_0,N_1)$.
\end{lemma}
\begin{proof}
We shall only consider the case when $\vec u$ is a weak solution of $L \vec u=\vec f$ since the proof of the other case is identical.
Throughout the proof, we denote by $C$ a constant depending on the prescribed parameters $n, m, \nu, p$ and also the numbers $\mu_0, N_1$ that appears in the condition \eqref{LH}.
As usual, the constant $C$ may vary from line to line.

Fix $R<R_{{max}}/4$ and let $\vec u$ be a weak solution of $L \vec u=\vec f$ in $\Omega_{4R}=\Omega_{4R}(x_0)$ vanishing on $\varSigma_{4R}$, where $\vec f \in L^p(\Omega_{4R})$ with $p \in (n/2,\infty]$.
Fix $x\in \Omega_R$ and $s \in (0,R]$.
We write $\vec u =\vec v+\vec w$ in $\Omega_s(x)$, where $\vec v \in W^{1,2}(\Omega_s(x))$ is a weak solution of $L \vec v=0$ in $\Omega_s(x)$ such that  $\vec v-\vec u \in W^{1,2}_0(\Omega_s(x))$. Notice that $\vec v$ vanishes on $\varSigma_s(x)$.
Then, \eqref{LH} implies that for $0<r<s$,
\begin{align*}
\int_{\Omega_r(x)} \abs{D\vec u}^2 &\le 2\int_{\Omega_r(x)} \abs{D\vec v}^2+2\int_{\Omega_r(x)} \abs{D\vec w}^2\\
&\le C \left(\frac{r}{s}\right)^{n-2+2\mu_0} \int_{\Omega_s(x)} \abs{D\vec v}^2+ 2\int_{\Omega_s(x)} \abs{D\vec w}^2\\
&\le C \left(\frac{r}{s}\right)^{n-2+2\mu_0} \int_{\Omega_s(x)} \abs{D\vec u}^2+C\int_{\Omega_s(x)} \abs{D\vec w}^2.
\end{align*}
Observe that $\vec w\in W^{1,2}_0(\Omega_s(x))$ and $\vec w$ is a weak solution of $L\vec w=\vec f$ in $\Omega_s(x)$. Therefore, we obtain
\[
\int_{\Omega_s(x)} \abs{D\vec w}^2 \le C \norm{\vec f}_{L^{2n/(n+2)}(\Omega_s(x))}^2.
\]
Choose $p_0\in (n/2,p)$ such that $\mu_1:=2-n/p_0<\mu_0$.
Then
\[
\norm{\vec f}_{L^{2n/(n+2)}(\Omega_s(x))}^2\le \norm{\vec f}_{L^{p_0}(\Omega_s(x))}^2 \abs{\Omega_s}^{1+2/n-2/p_0} \le C \norm{\vec f}_{L^{p_0}(\Omega_{2R})}^2 s^{n-2+2\mu_1}.
\]
By combining the above inequalities, we have for all $r<s\le R$
\[
\int_{\Omega_r(x)} \abs{D\vec u}^2 \le C\left(\frac{r}{s}\right)^{n-2+2\mu_0} \int_{\Omega_s(x)} \abs{D\vec u}^2+ C s^{n-2+2\mu_1}\norm{\vec f}_{L^{p_0}(\Omega_{2R})}^2.
\]
A well known iteration argument (see e.g., \cite[\S III.2]{Gi83}) yields that for all $r\in (0,R]$ and $x\in \Omega_R$, we have
\begin{equation}    \label{eq3.2}
\int_{\Omega_r(x)} \abs{D\vec u}^2 \le C\left(\frac{r}{R}\right)^{n-2+2\mu_1}\int_{\Omega_{2R}} \abs{D\vec u}^2+C r^{n-2+2\mu_1}\norm{\vec f}_{L^{p_0}(\Omega_{2R})}^2.
\end{equation}

Let $\tilde{\vec u}$ be an extension of $\vec u$ by zero on $B_{2R}\setminus \Omega_{2R}$.
Notice that $\tilde{\vec u}\in W^{1,2}(B_{2R})$ and $D \tilde{\vec u}= \chi_{\Omega_{2R}} D \vec u $ in $B_{2R}$.
Then by the Poincar\'e's inequality and \eqref{eq3.2}, we find that for all $r\in (0,R]$ and $x\in B_R$, we have
\begin{equation}    \label{eq3.3}
\int_{B_r(x)} \Abs{\tilde{\vec u}-\tilde{\vec u}_r}^2 \le C r^{n+2\mu_1}\left( R^{-n+2-2\mu_1} \norm{D\vec u}_{L^2(\Omega_{2R})}^2+\norm{\vec f}_{L^{p_0}(\Omega_{2R})}^2\right).
\end{equation}
Then it follows from \eqref{eq3.3} and H\"older's inequality that
\[
[\tilde{\vec u}]^2_{C^{\mu_1}(B_R)} \le C R^{-n+2-2\mu_1}\norm{D\vec u}_{L^2(\Omega_{2R})}^2+ CR^{4-2\mu_1-2n/p}\norm{\vec f}_{L^p(\Omega_{2R})}^2.
\]
Therefore, we obtain
\begin{align*}
\norm{\vec u}_{L^\infty(\Omega_{R/2})}^2 &\le C R^{2\mu_1} [\tilde{\vec u}]_{C^{\mu_1}(B_R)}^2+C R^{-n}\norm{\tilde{\vec u}}_{L^2(B_R)}^2 \\
&\le C R^{-n+2}\norm{D\vec u}_{L^2(\Omega_{2R})}^2+C R^{4-2n/p}\norm{\vec f}_{L^p(\Omega_{2R})}^2+C R^{-n}\norm{\vec u}_{L^2(\Omega_{R})}^2.
\end{align*}
Recall that $\vec u$ vanishes on $\varSigma_{4R}$.
By the Caccioppoli's inequality, we derive
\begin{align*}
\norm{D \vec u}_{L^2(\Omega_{2R})}^2 &\leq C R^{-2} \norm{\vec u}_{L^2(\Omega_{4R})}^2+C\norm{\vec f}_{L^{2n/(n+2)}(\Omega_{2R})}^2\\
&\leq C R^{-2} \norm{\vec u}_{L^2(\Omega_{4R})}^2+C R^{2+n-2n/p} \norm{\vec f}_{L^p(\Omega_{4R})}^2.
\end{align*}
By combining the above two inequalities and replacing $R$ by $R/4$, we obtain
\[
\norm{\vec u}_{L^\infty(\Omega_{R/8})} \le C R^{-n/2}\norm{\vec u}_{L^2(\Omega_R)}+ CR^{2-n/p} \norm{\vec f}_{L^p(\Omega_R)}.
\]
Finally, the above inequality together with a standard covering argument yields \eqref{eq6.01ap}.
The proof is complete.
\end{proof}

\begin{acknowledgment}
We thank the referee for useful comments.
Kyungkeun Kang was supported by the  Korean Research Foundation Grant (MOEHRD, Basic Research Promotion Fund, KRF-2008-331-C00024) and the National Research Foundation of
Korea(NRF) funded by the Ministry of Education, Science and Technology (2009-0088692).
Kyungkeun Kang appreciates the hospitality of the Department of Computational Science and Engineering,  Yonsei University.
Seick Kim was supported by the Korea Science and Engineering Foundation grant (MEST, No.~R01-2008-000-20010-0) and also by WCU(World Class University) program through the Korea Science and Engineering Foundation (MEST, No.~R31-2008-000-10049-0).
\end{acknowledgment}


\end{document}